\documentclass[notitlepage]{revtex4-1}

\usepackage{amssymb}
\usepackage{amsfonts}
\usepackage{amsmath}
\usepackage{amsthm}
\usepackage{mathtools}
\usepackage{extarrows}
\usepackage{graphicx}
\usepackage{mathrsfs}
\usepackage{dsfont}
\usepackage{amscd}
\usepackage{multirow}
\usepackage[all]{xy}
\usepackage[scaled]{helvet} 
\usepackage{courier} 
\normalfont
\usepackage[T1]{fontenc}
\usepackage{calligra}
\usepackage{verbatim}
\usepackage[usenames,dvipsnames]{color}
\usepackage[colorlinks=true,linkcolor=Blue,citecolor=Violet]{hyperref}

\newcommand{\N}{{\mathds{N}}}
\newcommand{\Z}{{\mathds{Z}}}

\newcommand{\R}{{\mathds{R}}}
\newcommand{\C}{{\mathds{C}}}
\newcommand{\T}{{\mathds{T}}}

\newcommand{\D}{{\mathfrak{D}}}
\newcommand{\A}{{\mathfrak{A}}}
\newcommand{\B}{{\mathfrak{B}}}

\newcommand{\Lip}{{\mathsf{L}}}
\newcommand{\Hilbert}{{\mathscr{H}}}

\newcommand{\dist}{{\mathsf{dist}}}

\newcommand{\qpropinquity}[1]{{\mathsf{\Lambda}_{#1}}}

\newcommand{\Kantorovich}[1]{{\mathsf{mk}_{#1}}}

\newcommand{\Haus}[1]{{\mathsf{Haus}_{#1}}}

\newcommand{\StateSpace}{{\mathscr{S}}}

\newcommand{\mongekant}{{Mon\-ge-Kan\-to\-ro\-vich metric}}

\newcommand{\Lqcms}{{\JLL} quan\-tum com\-pact me\-tric spa\-ce}

\newcommand{\unit}{1}

\newcommand{\sa}[1]{{\mathfrak{sa}\left({#1}\right)}}

\newcommand{\inner}[2]{{\left<{#1},{#2}\right>}}
\newcommand{\JLL}{Lei\-bniz}

\newcommand{\dom}[1]{{\operatorname*{dom}({#1})}}

\newcommand{\diam}[2]{{\mathrm{diam}\left({#1},{#2}\right)}}

\newcommand{\bridgereach}[2]{{\varrho\left({#1}\middle|{#2}\right)}}
\newcommand{\bridgeheight}[2]{{\varsigma\left({#1}\middle|{#2}\right)}}

\newcommand{\bridgelength}[2]{{\lambda\left({#1}\middle|{#2}\right)}}

\newcommand{\bridgenorm}[2]{{\mathsf{bn}_{ {#1}  }\left({#2}\right)}}

\newcommand{\Jordan}[2]{{{#1}\circ{#2}}} 
\newcommand{\Lie}[2]{{\left\{{#1},{#2}\right\}}} 

\newcommand{\alg}[1]{{\mathfrak{#1}}}

\newcommand{\opnorm}[1]{{\left|\mkern-1.5mu\left|\mkern-1.5mu\left| {#1} \right|\mkern-1.5mu\right|\mkern-1.5mu\right|}}
\newcommand{\GH}{{\mathrm{GH}}}

\theoremstyle{plain}
\newtheorem{theorem}{Theorem}[section]

\newtheorem{corollary}[theorem]{Corollary}

\newtheorem{lemma}[theorem]{Lemma}
\newtheorem{proposition}[theorem]{Proposition}

\newtheorem{theorem-definition}[theorem]{Theorem-Definition}

\theoremstyle{definition}
\newtheorem{definition}[theorem]{Definition}

\newtheorem{notation}[theorem]{Notation}

\newtheorem{convention}[theorem]{Convention}

\theoremstyle{remark}
\newtheorem*{acknowledgement}{Acknowledgement}

\renewcommand{\geq}{\geqslant}
\renewcommand{\leq}{\leqslant}

\numberwithin{equation}{section}

\begin{document}

\title{Curved Noncommutative Tori as Leibniz Quantum Compact Metric Spaces }
\author{Fr\'{e}d\'{e}ric Latr\'{e}moli\`{e}re}
\email[Email Address: ]{frederic@math.du.edu}
\homepage[Home Page: ]{http://www.math.du.edu/\symbol{126}frederic}
\affiliation{Department of Mathematics \\ University of Denver \\ Denver CO 80208}

\date{\today}


\begin{abstract}
We prove that curved noncommutative tori are {\Lqcms s} and that they form a continuous family over the group of invertible matrices with entries in the image of the quantum tori for the conjugation by modular conjugation operator in the regular representation, when this group is endowed with a natural length function.
\end{abstract}

\maketitle

\setcounter{secnumdepth}{1}
\tableofcontents


\section{Introduction}

The quantum Gromov-Hausdorff propinquity \cite{Latremoliere13,Latremoliere13b} provides a natural framework for the study of metrics perturbations in noncommutative geometry by extending the Gromov-Hausdorff distance \cite{Gromov81} to a metric, up to isometric isomorphism, on the class of {\Lqcms s}. Such metric perturbations may arise, for instance, in the study of conformal noncommutative geometry \cite{Connes08, Ponge14}, and may be of physical interests when studying noncommutative space-times and the fluctuations of an underlying metrics, for instance in relation with problems in quantum gravity.

The curved noncommutative tori, introduced and studied by D\k{a}browski and Sitarz \cite{Sitarz13, Sitarz15}, provide very interesting examples of perturbations of well-understood quantum metric spaces: the quantum tori, with the metric structure induced by their standard spectral triples. The original motivation in \cite{Sitarz13} was, at least in part, the study of the notion of curvature in noncommutative geometry. We propose in this paper to study the continuity of the family of curved noncommutative tori for the quantum Gromov-Hausdorff propinquity with respect to a natural topology on the parameter space used to deform the metrics. More specifically, the metrics on curved noncommutative tori arise from spectral triples constructed from noncommutative ``elliptic'' operators whose coefficients are chosen in the image of the quantum tori for the map $\mathrm{Ad}_J$ where $J$ is the modular conjugation operator in the left regular representation, i.e. some C*-subalgebra of the commutant of the quantum tori in its left regular representation. For a given noncommutative torus $\A$, the parameter space for these perturbations is thus given by a group of invertible matrices whose entries are in a C*-subalgebra of the commutant of the image of $\A$ in its regular representation, and we shall see that a natural length function on this group renders the family of perturbations continuous for the quantum Gromov-Hausdorff propinquity.

One step in this process, in particular, is to prove that curved noncommutative tori are indeed {\Lqcms s}. This step is folded in the proof of the main theorem on the continuity of the family of curved noncommutative tori, since it shares estimates with our main argument.

Our hope is that understanding such fundamental examples will expand our understanding of our quantum propinquity, and of the metric properties of quantum spaces. Our results in this paper contrast with other types of deformations, most notably conformal deformations, which we studied in \cite{Latremoliere15b}. Indeed, for conformal deformations, the parameter space consists of (differentiable) invertible elements of the quantum torus itself, and thus the continuity of conformal deformation of quantum tori involve a length function on the parameter space which involve a first-order quantity. In the present paper, such a condition would not be meaningful, and thankfully is not needed either. On the other hand, estimates needed for the proof of our main theorem requires some care and are different from the computations in \cite{Latremoliere15b} for the conformal case.

\bigskip

Noncommutative metric geometry proposes to study noncommutative analo\-gues of the algebras of Lipschitz functions over metric spaces using techniques inspired by metric geometry \cite{Gromov} and motivated by problems in mathematical physics. Thus, the basic structure of this theory is called a quantum metric space. We shall focus on the setting of \emph{compact} quantum metric spaces in this paper, which was introduced by Rieffel in \cite{Rieffel98a,Rieffel99}, after some initial steps by Connes \cite{Connes89}. The locally compact setting is significantly more involved, and studied in \cite{Latremoliere05b,Latremoliere12b}, and all our examples in this paper will be compact quantum metric spaces. In fact, they will fit in the following framework.

\begin{notation}
We shall use two notations for norms. When $E$ is a normed vector space, then its norm will be denoted by $\|\cdot\|_E$ by default. The algebra of all continuous linear endomorphisms of $E$ will be denoted by $\alg{B}(E)$, and it carries the operator norm $\|\cdot\|_{\B(E)}$. However, to lighten our notations, we will denote $\|\cdot\|_{\alg{B}(E)}$ simply as $\opnorm{\cdot}_E$.
\end{notation}

\begin{notation}
Let $\A$ be a unital C*-algebra. The unit of $\A$ will be denoted by $\unit_\A$. The state space of $\A$ will be denoted by $\StateSpace(\A)$ while the self-adjoint part of $\A$ will be denoted by $\sa{\A}$. 
\end{notation}

\begin{convention}
When $\Lip$ is a seminorm defined on some dense subset $F$ of a vector space $E$, we will implicitly extend $\Lip$ to $E$ by setting $\Lip(e) = \infty$ whenever $e \not \in F$.
\end{convention}

\begin{definition}[\cite{Latremoliere13}]\label{lcqms-def}
A {\Lqcms} $(\A,\Lip)$ is an ordered pair where $\A$ is a unital C*-algebra and $\Lip$ is a seminorm defined on some dense Jordan-Lie subalgebra $\dom{\Lip}$ of $\sa{\A}$ such that:
\begin{enumerate}
\item $\left\{a\in\sa{\A} : \Lip(a) = 0 \right\} = \R\unit_\A$,
\item the {\mongekant} $\Kantorovich{\Lip}$ defined, for any two states $\varphi,\psi \in \StateSpace(\A)$ by:
\begin{equation*}
\Kantorovich{\Lip}(\varphi,\psi) = \sup\left\{ |\varphi(a)-\psi(a)| : a\in\dom{\Lip}, \Lip(a)\leq 1 \right\}
\end{equation*}
metrizes the weak* topology of $\StateSpace(\A)$,
\item for all $a,b \in \dom{\Lip}$ we have:
\begin{equation*}
\Lip\left(\Jordan{a}{b}\right), \Lip\left(\Lie{a}{b}\right) \leq \Lip(a)\|b\|_\A + \|a\|_\A\Lip(b)\text{,}
\end{equation*}
where for all $a,b \in \sa{\A}$, we denote the Jordan product $\frac{ab+ba}{2}$ of $a$ and $b$ by $\Jordan{a}{b}$ and the Lie product $\frac{ab-ba}{2i}$ of $a$ and $b$ by $\Lie{a}{b}$,
\item $\{a\in\sa{\A} : \Lip(a)\leq 1 \}$ is closed for $\|\cdot\|_\A$.
\end{enumerate}
A seminorm $\Lip$ satisfying Assertions (1) and (2) is called a Lip-norm.
\end{definition}

The fundamental examples of a {\Lqcms} are given by the pairs of the form $(C(X),\Lip_{\mathrm{d}})$ where $(X,\mathrm{d})$ is a compact metric space and $\Lip_{\mathrm{d}}$ is the usual Lipschitz seminorm induced by $\mathrm{d}$. The metric $\Kantorovich{\Lip_{\mathrm{d}}}$ is then the original metric introduced by Kantorovich \cite{Kantorovich40,Kantorovich58} in his study of the Monge transportation problem. One particularly important property of the {\mongekant} is that it induces the weak* topology on the set of Borel regular probability measures on the underlying compact space. This property was chosen, in Assertion (2) of Definition (\ref{lcqms-def}), as a foundation for the theory of quantum compact metric spaces. The {\mongekant}, and its geometric properties, have been studied to great length \cite{Villani09}, and one goal of noncommutative metric geometry is to understand this metric in the quantum context. We refer to \cite{Rieffel05, Latremoliere15b} and the references therein for a discussion on the motivations, basic properties, and examples of {\Lqcms s}.

An important family of examples of {\Lqcms s} is given by the family of the quantum tori and their natural spectral triples \cite{Rieffel98a}. More precisely, if $(\A,\Hilbert,D)$ is a spectral triple over a unital C*-algebra $\A$ \cite{Connes89, Connes}, then one may define the seminorm $\Lip : a\in\sa{\A}\mapsto \opnorm{[D,a]}_\Hilbert$ and ask when such a seminorm becomes a Lip-norm. While this question is not settled in general, the special case of the quantum tori is understood \cite{Rieffel98a} --- and in fact, many choices of spectral triples provide different and interesting quantum metric structures \cite{Rieffel98a, Rieffel02}. We shall focus in this paper on spectral triples obtained by transporting the differential structure of the torus to the quantum torus via the dual action, and certain specific perturbations of these spectral triples studied in \cite{Sitarz13, Sitarz15}. We will provide a detailed description of these triples in this paper.

When defining perturbations of metrics or spectral triples in noncommutative geometry, the general strategy in the literature seems to employ an algebraic notion of perturbations, so to speak: the Dirac operator of the triple is replaced by a modified operator given by an algebraic expression which one deems close in form --- close in spirit, informally. It is desirable to make the notion of perturbation more precise: if quantified, then one could start working with families of perturbations in a more analytic manner and precisely state what being ``close'' to the original metric means for a perturbation.

The quantum Gromov-Hausdorff propinquity provides a means for such a quantification. The search for a noncommutative analogue of the Gromov-Hausdorff distance in noncommutative geometry has met with various challenges; the first successful construction of such an analogue is due to Rieffel \cite{Rieffel00}. It was soon followed by other attempts \cite{kerr02,li03,kerr09} designed to address various potential issues with Rieffel's distance. Yet, progress in noncommutative metric geometry eventually revealed the important role of the Leibniz property of Lip-norms as a tool to connect quantum metric structures (the Lip-norms) and quantum topological structures (the C*-algebras), and no Gromov-Hausdorff metric was particularly adequate for this particular type of structures. We refer to \cite{Latremoliere13,Latremoliere13b} for a discussion on this topic. 

As an answer to many of the challenges raised by constructing a noncommutative Gromov-Hausdorff distance suitable for the study of C*-algebraic properties and based on Leibniz Lip-norms, we constructed the quantum propinquity in \cite{Latremoliere13}. Later on, we also constructed a complete distance on the class of {\Lqcms s}, called the dual Gromov-Hausdorff propinquity \cite{Latremoliere13b,Latremoliere14}, and noted that the quantum propinquity can be seen as an important special case of the construction in \cite{Latremoliere13}. Both our Gromov-Hausdorff propinquity metrics extend the Gromov-Hausdorff topology from the class of classical compact metric spaces to the noncommutative realm, and both metrics are zero between two {\Lqcms s} if and only if they are isometrically isomorphic. Both these metrics are designed specifically to handle Leibniz Lip-norms and seem to address many of the original difficulties in this program. In fact, the quantum propinquity is a special case of the construction of the dual propinquity, and in particular, dominates it. Thus all the results in this paper, which are proven for the quantum propinquity, are also valid for the dual propinquity.

Thus, equipped with our noncommutative analogue of the Gromov-Hausdorff distance, we can now quantify how far two {\Lqcms s} are from each other, and in particular study the continuity of various natural families of {\Lqcms s}. For instance, in \cite{Latremoliere13c}, building on \cite{Latremoliere05}, we proved that quantum tori form a continuous family for a natural quantum metric, and moreover that quantum tori are limits, for the quantum propinquity, of finite dimensional {\Lqcms s}. In \cite{Latremoliere15}, we established a sufficient condition for a {\Lqcms} to have finite dimensional approximations for the dual propinquity (while relaxing a little bit the Leibniz property), and we proved a noncommutative generalization of Gromov's compactness theorem. In \cite{Rieffel10c,Rieffel11,Rieffel15}, Rieffel showed that the C*-algebras of continuous functions on coadjoint orbits of semisimple Lie groups, endowed with a classical Lipschitz seminorm from an invariant metric, are limits of matrix algebras for the quantum propinquity. Many of these results are motivated by the literature in mathematical physics.

In this work, we propose to use the quantum propinquity to derive explicit bounds on the distance between different curved spectral triples on a given quantum torus. Thus we provide a quantifiable meaning to the notion of metric perturbation. Paired with our results in \cite{Latremoliere13b,Latremoliere15b}, we thus can vary both the quantum metric structure and the deformation parameters of quantum tori continuously for the quantum propinquity.

Our paper begins with a brief review of the construction of the quantum propinquity. This construction is found in \cite{Latremoliere13} and our presentation here is purposefully elliptic, designed to provide a notational and conceptual framework for the rest of the paper. We then prove a first, simple result on a first type of deformation of metrics from spectral triples, which are formally similar to conformal deformations, yet simpler to manage. Our second section contains our main theorem, where we obtain an explicit bound on the quantum propinquity between two curved noncommutative tori. 

\begin{acknowledgement}
We thank A. Sitarz for having brought this problem to our attention. We also thank the referee of this paper for very helpful comments.
\end{acknowledgement}

\section{A continuous field of {\Lqcms s}}

\subsection{The quantum Gromov-Hausdorff propinquity}

We begin this section with a brief exposition of the quantum Gromov-Hausdorff propinquity defined on the class of {\Lqcms s}. The construction of this metric is carried away in \cite{Latremoliere13}: the next few pages are intended as a means to ease the exposition of this paper, and we fully refer to \cite{Latremoliere13} for the motivation, context, actual construction, the proof of Theorem-Definition (\ref{def-thm}) and many other properties of our metric.

The idea behind the quantum propinquity is to define a means to compare how close two {\Lqcms s} are by comparing how close we can make the balls of their Lip-norms for some special seminorms constructed out of C*-algebras. As a first step, we formalize a notion of embedding, called \emph{a bridge}.

\begin{definition}
The \emph{$1$-level set} $\StateSpace_1(\D|\omega)$ of an element $\omega$ of a unital C*-algebra $\D$ is:
\begin{equation*}
\left\{ \varphi \in \StateSpace(\D) : \varphi((1-\omega^\ast\omega))=\varphi((1-\omega \omega^\ast)) = 0 \right\}\text{.}
\end{equation*}
\end{definition}

\begin{definition}
A \emph{bridge} from a unital C*-algebra $\A$ to a unital C*-algebra $\B$ is a quadruple $(\D,\pi_\A,\pi_\B,\omega)$ where:
\begin{enumerate}
\item $\D$ is a unital C*-algebra,
\item the element $\omega$, called the \emph{pivot} of the bridge, satisfies $\omega\in\D$ and $\StateSpace_1(\D|\omega) \not=\emptyset$,
\item $\pi_\A : \A\hookrightarrow \D$ and $\pi_\B : \B\hookrightarrow\D$ are unital *-monomorphisms.
\end{enumerate}
\end{definition}

There always exists a bridge between any two arbitrary {\Lqcms s} \cite{Latremoliere13}. A bridge allows us to define a numerical quantity which estimates, for this given bridge, how far our {\Lqcms s} are. This quantity, called the length of the bridge, is constructed using two other quantities we now define. 

In the next few definitions, for a unital *-morphism $\pi:\A\rightarrow\D$ between two unital C*-algebras $\D$ and $\A$, we denote the dual map $\mu\in\D^\ast\mapsto \mu\circ\pi\in\A^\ast$ by $\pi^\ast$ and we note that $\pi^\ast$ maps states to states. We also denote by $\Haus{\mathrm{d}}$ the Hausdorff (pseudo)distance induced by a (pseudo)distance $\mathrm{d}$ on the compact subsets of a (pseudo)metric space $(X,\mathrm{d})$ \cite{Hausdorff}.

The height of a bridge assesses the error we make by replacing the state spaces of the {\Lqcms s} with the image of the $1$-level set of the pivot of the bridge, using the ambient {\mongekant}. 

\begin{definition}
Let $(\A,\Lip_\A)$ and $(\B,\Lip_\B)$ be two {\Lqcms s}. The \emph{height} $\bridgeheight{\gamma}{\Lip_\A,\Lip_\B}$ of a bridge $\gamma = (\D,\pi_\A,\pi_\B,\omega)$ from $\A$ to $\B$, and with respect to $\Lip_\A$ and $\Lip_\B$, is given by:
\begin{equation*}
\max\left\{ \Haus{\Kantorovich{\Lip_\A}}(\StateSpace(\A), \pi_\A^\ast(\StateSpace_1(\D|\omega))), \Haus{\Kantorovich{\Lip_\B}}(\StateSpace(\B), \pi_\B^\ast(\StateSpace_1(\D|\omega))) \right\}\text{.}
\end{equation*}
\end{definition}

The second quantity measures how far apart the images of the balls for the Lip-norms are in $\A\oplus\B$; to do so, they use a seminorm on $\A\oplus\B$ built using the bridge:
\begin{definition}
Let $(\A,\Lip_\A)$ and $(\B,\Lip_\B)$ be two unital C*-algebras. The \emph{bridge seminorm} $\bridgenorm{\gamma}{\cdot}$ of a bridge $\gamma = (\D,\pi_\A,\pi_\B,\omega)$ from $\A$ to $\B$ is the seminorm defined on $\A\oplus\B$ by:
\begin{equation*}
\bridgenorm{\gamma}{a,b} = \|\pi_\A(a)\omega - \omega\pi_\B(b)\|_\D
\end{equation*}
for all $(a,b) \in \A\oplus\B$.
\end{definition}

We implicitly identify $\A$ with $\A\oplus\{0\}$ and $\B$ with $\{0\}\oplus\B$ in $\A\oplus\B$ in the next definition, for any two spaces $\A$ and $\B$.

\begin{definition}
Let $(\A,\Lip_\A)$ and $(\B,\Lip_\B)$ be two {\Lqcms s}. The \emph{reach} $\bridgereach{\gamma}{\Lip_\A,\Lip_\B}$ of a bridge $\gamma = (\D,\pi_\A,\pi_\B,\omega)$ from $\A$ to $\B$, and with respect to $\Lip_\A$ and $\Lip_\B$, is given by:
\begin{equation*}
\Haus{\bridgenorm{\gamma}{\cdot}}\left( \left\{a\in\sa{\A} : \Lip_\A(a)\leq 1\right\} , \left\{ b\in\sa{\B} : \Lip_\B(b) \leq 1 \right\}  \right) \text{.}
\end{equation*}
\end{definition}

We thus choose a natural synthetic quantity to summarize the information given by the height and the reach of a bridge:

\begin{definition}
Let $(\A,\Lip_\A)$ and $(\B,\Lip_\B)$ be two {\Lqcms s}. The \emph{length} $\bridgelength{\gamma}{\Lip_\A,\Lip_\B}$ of a bridge $\gamma = (\D,\pi_\A,\pi_\B,\omega)$ from $\A$ to $\B$, and with respect to $\Lip_\A$ and $\Lip_\B$, is given by:
\begin{equation*}
\max\left\{\bridgeheight{\gamma}{\Lip_\A,\Lip_\B}, \bridgereach{\gamma}{\Lip_\A,\Lip_\B}\right\}\text{.}
\end{equation*}
\end{definition}

While a natural approach, defining the quantum propinquity as the infimum of the length of all possible bridges between two given {\Lqcms s} does not lead to a distance, as the triangle inequality may not be satisfied. Instead, a more subtle road must be taken, as exposed in details in \cite{Latremoliere13}. The following theorem hides these complications and provide a summary of the conclusions of \cite{Latremoliere13} relevant for our work:

\begin{theorem-definition}[\cite{Latremoliere13}]]\label{def-thm}
Let $\mathcal{L}$ be the class of all {\Lqcms s}. There exists a class function $\qpropinquity{}$ from $\mathcal{L}\times\mathcal{L}$ to $[0,\infty) \subseteq \R$ such that:
\begin{enumerate}
\item for any $(\A,\Lip_\A), (\B,\Lip_\B) \in \mathcal{L}$ we have:
\begin{equation*}
0\leq \qpropinquity{}((\A,\Lip_\A),(\B,\Lip_\B)) \leq \max\left\{\diam{\StateSpace(\A)}{\Kantorovich{\Lip_\A}}, \diam{\StateSpace(\B)}{\Kantorovich{\Lip_\B}}\right\}\text{,}
\end{equation*}
\item for any $(\A,\Lip_\A), (\B,\Lip_\B) \in \mathcal{L}$ we have:
\begin{equation*}
\qpropinquity{}((\A,\Lip_\A),(\B,\Lip_\B)) = \qpropinquity{}((\B,\Lip_\B),(\A,\Lip_\A))\text{,}
\end{equation*}
\item for any $(\A,\Lip_\A), (\B,\Lip_\B), (\alg{C},\Lip_{\alg{C}}) \in \mathcal{L}$ we have:
\begin{equation*}
\qpropinquity{}((\A,\Lip_\A),(\alg{C},\Lip_{\alg{C}})) \leq \qpropinquity{}((\A,\Lip_\A),(\B,\Lip_\B)) + \qpropinquity{}((\B,\Lip_\B),(\alg{C},\Lip_{\alg{C}}))\text{,}
\end{equation*}
\item for all  for any $(\A,\Lip_\A), (\B,\Lip_\B) \in \mathcal{L}$ and for any bridge $\gamma$ from $\A$ to $\B$, we have:
\begin{equation*}
\qpropinquity{}((\A,\Lip_\A), (\B,\Lip_\B)) \leq \bridgelength{\gamma}{\Lip_\A,\Lip_\B}\text{,}
\end{equation*}
\item for any $(\A,\Lip_\A), (\B,\Lip_\B) \in \mathcal{L}$, we have:
\begin{equation*}
\qpropinquity{}((\A,\Lip_\A),(\B,\Lip_\B)) = 0
\end{equation*}
if and only if $(\A,\Lip_\A)$ and $(\B,\Lip_\B)$ are isometrically isomorphic, i.e. if and only if there exists a *-isomorphism $\pi : \A \rightarrow\B$ with $\Lip_\B\circ\pi = \Lip_\A$, or equivalently there exists a *-isomorphism $\pi : \A \rightarrow\B$ whose dual map $\pi^\ast$ is an isometry from $(\StateSpace(\B),\Kantorovich{\Lip_\B})$ into $(\StateSpace(\A),\Kantorovich{\Lip_\A})$,

\item if $\Xi$ is a class function from $\mathcal{L}\times \mathcal{L}$ to $[0,\infty)$ which satisfies Properties (2), (3) and (4) above, then:
\begin{equation*}
\Xi((\A,\Lip_\A), (\B,\Lip_\B)) \leq \qpropinquity{}((\A,\Lip_\A),(\B,\Lip_\B))
\end{equation*}
for all $(\A,\Lip_\A)$ and $(\B,\Lip_\B)$ in $\mathcal{L}$.
\end{enumerate}
\end{theorem-definition}

Thus the quantum propinquity is the largest pseudo-distance on the class of {\Lqcms s} which is bounded above by the length of any bridge between its arguments; the remarkable conclusion of \cite{Latremoliere13} is that this pseudo-metric is in fact a metric up to isometric isomorphism. Moreover, we showed in \cite{Latremoliere13} that we can compare the quantum propinquity to natural metrics.

\begin{theorem}[\cite{Latremoliere13}]
If $\dist_q$ is Rieffel's quantum Gromov-Hausdorff distance \cite{Rieffel00}, then for any pair $(\A,\Lip_\A)$ and $(\B,\Lip_\B)$ of {\Lqcms s}, we have:
\begin{equation*}
\dist_q((\A,\Lip_\A),(\B,\Lip_\B)) \leq \qpropinquity{}((\A,\Lip_\A),(\B,\Lip_\B)) \text{.}
\end{equation*}
Moreover, for any compact metric space $(X,\mathrm{d}_X)$, let $\Lip_{\mathrm{d}_X}$ be the Lipschitz seminorm induced on the C*-algebra $C(X)$ of $\C$-valued continuous functions on $X$ by $\mathrm{d}_X$. Note that $(C(X),\Lip_{\mathrm{d}_X})$ is a {\Lqcms}. Let $\mathfrak{C}$ be the class of all compact metric spaces. For any $(X,\mathrm{d}_x), (Y,\mathrm{d_Y}) \in \mathfrak{C}$, we have:
\begin{equation*}
\qpropinquity{}(C(X,\mathrm{d}_X), C(Y,\mathrm{d}_Y)) \leq \GH((X,\mathrm{d}_X),(Y,\mathrm{d}_Y))
\end{equation*}
where $\GH$ is the Gromov-Hausdorff distance \cite{Gromov81, Gromov}. 

Furthermore, the class function $\Upsilon : (X,\mathrm{d}_X) \in \mathfrak{C} \mapsto (C(X),\Lip_{\mathrm{d}_X})$ is an homeomorphism, where the topology on $\mathfrak{C}$  is given by the Gromov-Hausdorff distance $\GH$, and the topology on the image of $\Upsilon$ (as a subclass of the class of all {\Lqcms s}) is given by the quantum propinquity $\qpropinquity{}$.
\end{theorem}

As we noted, the construction and many more information on the quantum Gromov-Hausdorff propinquity can be found in our original paper \cite{Latremoliere13} on this topic, as well as in our survey \cite{Latremoliere15b}. Two very important examples of nontrivial convergences for the quantum propinquity are given by quantum tori and their finite dimensional approximations \cite{Latremoliere13c} and by matrix approximations of the C*-algebras of coadjoint orbits for semisimple Lie groups \cite{Rieffel10c, Rieffel11,Rieffel15}. Moreover, the quantum propinquity is, in fact, a special form of the dual Gromov-Hausdorff propinquity \cite{Latremoliere13b, Latremoliere14, Latremoliere14b, Latremoliere15}, which is a complete metric, up to isometric isomorphism, on the class of {\Lqcms s}, and which extends the topology of the Gromov-Hausdorff distance as well. Thus, as the dual propinquity is dominated by the quantum propinquity \cite{Latremoliere13b}, we conclude that \emph{all the convergence results in this paper are valid for the dual Gromov-Hausdorff propinquity as well.}

\bigskip

We now recall \cite[Lemma 3.79]{Latremoliere15b}, which is a simple tool to establish some estimates on the quantum propinquity. It was used in \cite[section 3.7]{Latremoliere15b} to study certain perturbations of metrics defined by spectral triples, and we will employ it to a similar end in this paper.

\begin{lemma}[\cite{Latremoliere15b}]\label{propinquity-estimate-lemma}
Let $(\A,\Lip_\A)$ and $(\B,\Lip_\B)$ be two {\Lqcms s}. If $\gamma = (\D,\pi_\A,\pi_\B,\omega)$ is a bridge from $\A$ to $\B$ and if there exists $\delta > 0$ such that:
\begin{enumerate}
\item for all $a\in\sa{\A}$ there exists $b\in\sa{\B}$ such that:
\begin{equation*}
\max\{ \|\pi_\A(a)\omega - \omega\pi_\B(b)\|_\D, |\Lip_\A(a) - \Lip_\B(b)| \} \leq \delta \Lip_\A(a)\text{,}
\end{equation*}
\item for all $b\in\sa{\A}$ there exists $a\in\sa{\B}$ such that:
\begin{equation*}
\max\{ \|\pi_\A(a)\omega - \omega\pi_\B(b)\|_\D, |\Lip_\A(a) - \Lip_\B(b)| \} \leq \delta \Lip_\B(b)\text{,}
\end{equation*}
then:
\begin{multline*}
\qpropinquity{}((\A,\Lip_\A),(\B,\Lip_\B)) \leq
\max\left\{ \delta\left(1 + \frac{1}{2} \max\{\diam{\StateSpace(\A)}{\Kantorovich{\Lip_\A}}, \diam{\StateSpace(\B)}{\Kantorovich{\Lip_\B}} \right), \bridgeheight{\gamma}{\Lip_\A,\Lip_\B} \right\}\text{.}
\end{multline*}
\end{enumerate}
\end{lemma}

Of particular relevance in this paper is the observation that if the pivot element $\omega$ of a bridge $\gamma = (\D,\pi,\rho,\omega)$ is the unit of $\D$, then the height of the bridge $\gamma$ is zero.

\subsection{Perturbations of the metric}

We study a particular form of perturbation of a quantum metric given by a spectral triple. Let $\A$ be a unital C*-algebra, $\pi$ a faithful nondegenerate representation of $\A$ on a Hilbert space $\Hilbert$, and $D$ a linear, possibly unbounded operator on $\Hilbert$, such that if we let:
\begin{equation*}
\Lip(a) = \opnorm{[D,\pi(a)]}_{\Hilbert}
\end{equation*}
for all $a\in\sa{\A}$ (allowing for $\Lip(a) = \infty$), then $(\A,\Lip)$ is a {\Lqcms}. A means to perturb the Lip-norm $\Lip$, involving invertible elements of $\sa{\A}$ and generalizing the idea of conformal deformations to noncommutative geometry \cite{Connes08, Ponge14}, was studied from our metric perspective in \cite[section 3.7]{Latremoliere15b}. Yet such perturbations involve twisted spectral triples.

To remain within the framework of spectral triples, a different approach is studied in this section, proposed by Sitarz and D\k{a}browski \cite{Sitarz13}. If we let $H$ be an invertible element in the commutant of $\pi(\A)$, then we may define:
\begin{equation*}
D_H = HDH\quad\text{ and }\quad \Lip_H : a \in \sa{\A} \longmapsto \opnorm{[D_H,\pi(a)]}_{\Hilbert}
\end{equation*}
and we may ask whether $(\A,\Lip_H)$ is a {\Lqcms}, and how far from $(\A,\Lip)$ it lies for the quantum propinquity. 

To answer this first question, we shall employ a result from Rieffel, established in \cite{Rieffel98a}, for which we present a quick and different proof for the reader's convenience.

\begin{lemma}[\cite{Rieffel98a}]\label{comp-lemma}
Let $\Lip$ be a Lip-norm on a unital C*-algebra $\A$ and let $\mathrm{S}$ be a seminorm defined on the same domain as $\Lip$. If there exists $\delta > 0$ such that for all $a\in\sa{\A}$:
\begin{equation*}
\Lip(a) \leq \delta \mathrm{S}(a)
\end{equation*}
then $\mathrm{S}$ is a Lip-norm on $\A$, and moreover:
\begin{equation}\label{comparison-eq2}
\diam{\StateSpace(\A)}{\Kantorovich{\mathrm{S}}} \leq \delta^{-1} \diam{\StateSpace(\A)}{\Kantorovich{\Lip}}\text{.}
\end{equation}
\end{lemma}

\begin{proof}
Note first that if $\mathrm{S}(a) = 0$ then $\Lip(a) = 0$ so $a \in \R\unit_\A$. Thus $\Kantorovich{\mathrm{S}}$ defines an extended metric on $\StateSpace(\A)$. Moreover:
\begin{equation}\label{comparison-eq}
\Kantorovich{\mathrm{S}} \leq \delta^{-1} \Kantorovich{\Lip}\text{.}
\end{equation}
Thus, $\Kantorovich{\mathrm{S}}$ is in fact a metric on $\StateSpace(\A)$, and moreover the topology $\tau'$ induced by $\Kantorovich{\mathrm{S}}$, which is Hausdorff, is coarser than the topology $\tau$ induced by $\Kantorovich{\Lip}$, which is compact. Thus $\tau = \tau'$, and we note $\tau$ is the weak* topology restricted to $\StateSpace(\A)$ by definition. Last, Inequality (\ref{comparison-eq}) implies Inequality (\ref{comparison-eq2}).
\end{proof}

Now, let $H, K$ be invertible, bounded linear operators on $\Hilbert$ which commutes with $\pi(\A)$. Then, for all $a\in\A$, if we denote $HDH$ by $D_H$ and $KDK$ by $D_K$, we have:

\begin{equation}\label{a-computation-eq1}
\begin{split}
[D_H, \pi(a)] &= H D H \pi(a) - \pi(a) H D H\\
&= H D\pi(a) H - H \pi(a) D H\\
&= H [D,\pi(a)] H \\
&= H K^{-1} [D_K, \pi(a)] K^{-1} H \text{,}
\end{split}
\end{equation}
where our computations are carried out on the domain of $D$. 

Hence, we can deduce the following result:

\begin{proposition}\label{simple-twist-prop}
Let $\A$ be a unital C*-algebra, $\pi$ a faithful nondegenerate representation of $\A$ on a Hilbert space $\Hilbert$, and $D$ a linear, possibly unbounded operator on $\Hilbert$. For any invertible bounded linear operator $H$ on $\Hilbert$ such that $H$ commutes with $\pi(\A)$, we set $D_H = H D H$ and for all $a\in\sa{\A}$, we set:
\begin{equation*}
\Lip_H(a) = \opnorm{[D_H,\pi(a)]}_{\Hilbert}\text{.}
\end{equation*}
The following assertions are then equivalent:
\begin{enumerate}
\item $(\A,\Lip)$ is a {\Lqcms},
\item $(\A,\Lip_H)$ is a {\Lqcms} for some bounded linear operator $H$ of $\Hilbert$ which commutes with $\pi(\A)$,
\item $(\A,\Lip_H)$ is a {\Lqcms} for all bounded linear operators $H$ of $\Hilbert$ which commutes with $\pi(\A)$.
\end{enumerate}
\end{proposition}

\begin{proof}
Fix an invertible operator $H$ in the commutant of $\pi(\A)$. By Inequality (\ref{a-computation-eq1}), we have $\Lip \leq \opnorm{H^{-1}}^2_{\Hilbert} \Lip_H$, so if $\Lip$ is a Lip-norm, then by Lemma (\ref{comp-lemma}), $\Lip_H$ is a Lip-norm as well. It is straightforward that $\Lip_H$ is closed and Leibniz.
Now, for any $K$ invertible in the commutant of $\pi(\A)$, again by Inequality (\ref{a-computation-eq1}), we have $\Lip_K \leq \opnorm{HK^{-1}}_\Hilbert\opnorm{KH^{-1}}_\Hilbert \Lip_H$, and thus $\Lip_K$ is also a Lip-norm. This result holds true for $H$ chosen to be the identity, and this completes our equivalences.
\end{proof}

We can moreover compute an estimate on the quantum propinquity between two metrics obtained in Proposition (\ref{simple-twist-prop}).

\begin{theorem}
Let $\A$ be a unital C*-algebra, $\pi$ a faithful nondegenerate representation of $\A$ on a Hilbert space $\Hilbert$, and $D$ a linear, possibly unbounded operator on $\Hilbert$. For any invertible bounded linear operator $H$ on $\Hilbert$ such that $H$ commutes with $\pi(\A)$, we set $D_H = H D H$ and for all $a\in\sa{\A}$, we set:
\begin{equation*}
\Lip_H(a) = \opnorm{[D_H,\pi(a)]}_\Hilbert\text{.}
\end{equation*}
If $(H_n)_{n\in\N}$ is a sequence of invertible linear operators on $\Hilbert$ so that:
\begin{equation*}
\lim_{n\rightarrow\infty} \opnorm{H H^{-1}_n - \unit_\D }_\Hilbert = 0 
\end{equation*}
then:
\begin{equation*}
\lim_{n\rightarrow\infty} \qpropinquity{}((\A,\Lip_{H_n}), (\A,\Lip_H)) = 0 \text{.}
\end{equation*}
\end{theorem}

\begin{proof}
Let $H, K$ be invertible bounded linear operators on $\Hilbert$, both chosen in the commutant of $\pi(\A)$. For any $a\in\dom{\Lip}$, we have:
\begin{equation}\label{another-computation-eq2}
\begin{split}
|\Lip_H(a) - \Lip_K(a)| &\leq \opnorm{H [D,\pi(a)] H - K [D,\pi(a)] K}_\Hilbert\\
&= \opnorm{HK^{-1}(K [D,\pi(a)] K) K^{-1} H - [D_K,\pi(a)]}_\Hilbert\\
&\leq \opnorm{HK^{-1} [D_K,\pi(a)] K^{-1}H - HK^{-1} [D_K,\pi(a)]}_\Hilbert \\
&\quad + \opnorm{ HK^{-1} [D_K,\pi(a)] -  [D_K,\pi(a)]}_\Hilbert\\
&\leq \opnorm{HK^{-1}}\Lip_K(a)\opnorm{1-K^{-1}H}_\Hilbert + \opnorm{1-HK^{-1}}_\Hilbert\Lip_K(a)\\
&\leq \left(\opnorm{HK^{-1}}_\Hilbert+1\right)\opnorm{1-K^{-1}H}_\Hilbert\Lip_K(a) \text{.}
\end{split}
\end{equation}
To ease notations, we set:
\begin{equation*}
\delta(K,H) = \left(\opnorm{HK^{-1}}_\Hilbert+1\right)\opnorm{1-K^{-1}H}_\Hilbert\text{.}
\end{equation*}

We note that our estimate is, in fact, symmetric in $H$ and $K$, i.e.
\begin{equation*}
|\Lip_H(a) - \Lip_K(a)| \leq \delta(H,K) \Lip_H(a) \text{.}
\end{equation*}

Now, if $\D$ is the C*-algebra of all bounded linear operators on $\Hilbert$, then $\gamma = (\D,\pi,\pi, \unit_\D)$ is a bridge from $\A$ to $\A$. Moreover, the height of $\gamma$ is zero. Thus, by Lemma (\ref{propinquity-estimate-lemma}) and Equation (\ref{another-computation-eq2}), we obtain:
\begin{equation*}
\qpropinquity{}((\A,\Lip_H), (\A,\Lip_K)) \leq \max\left\{\delta(H,K),\delta(K,H)\right\} \times \left( 1 + \max\{\diam{\StateSpace(\A)}{\Kantorovich{\Lip_H}}, \diam{\StateSpace(\A)}{\Kantorovich{\Lip_K}} \}\right)\text{.}
\end{equation*}

By Inequality (\ref{a-computation-eq1}), we also note that $\Lip_H \leq \opnorm{HK^{-1}}_\Hilbert\opnorm{KH^{-1}}_\Hilbert \Lip_K$ and thus:
\begin{equation*}
\diam{\StateSpace(\A)}{\Kantorovich{\Lip_K}} \leq \opnorm{H K^{-1}}_\Hilbert^{-1}\opnorm{KH^{-1}}_\Hilbert^{-1} \diam{\StateSpace(\A)}{\Kantorovich{\Lip_H}}\text{.}
\end{equation*}

We now put our estimates together. If $\left(\opnorm{1-H^{-1} H_n}_\Hilbert\right)_{n\in\N}$ converges to $0$, then $\left(\opnorm{H_n H^{-1}}_\Hilbert^{-1}\right)_{n\in\N}$ and $\left(\opnorm{H H_n^{-1}}^{-1}_\Hilbert\right)_{n\in\N}$ are bounded, say by some $m \geq 1$.  Thus:
\begin{equation*}
\qpropinquity{}((\A,\Lip_{H_n}),(\A,\Lip_H)) \leq \max\{ \delta(H,H_n), \delta(H_n,H) \} \left(1 + \frac{m^2}{2} \diam{\StateSpace(\A)}{\Kantorovich{\Lip_H}}\right)\text{.}
\end{equation*}

Our assumptions ensure as well that $(\opnorm{H_nH^{-1}}_\Hilbert)_{n\in\N}$ and $(\opnorm{H H_n^{-1}}_\Hilbert)_{n\in\N}$ are bounded, say by some $M\geq 0$. Since $\delta(H,H_n)\leq (1+M)\opnorm{1-H^{-1} H_n}_\Hilbert$ for all $n\in\N$, our assumption implies that $\left(\delta(H,H_n)\right)_{n\in\N}$ converges to $0$. Similarly, $\left(\delta(H_n,H)\right)_{n\in\N}$ converges to $0$. Thus our theorem is proven.
\end{proof}

\section{Curved Quantum Tori}

This section proves that the family of curved quantum tori introduced by D\k{a}browski and {S}itarz \cite{Sitarz13, Sitarz15} form a continuous family of {\Lqcms s} for the quantum propinquity. We will prove this result in a slightly more general setting, which we now describe.

Let $\alpha$ be an action of a Lie group $G$ on a C*-algebra $\A$. By \cite[Proposition 3.6.1]{Bratteli79}, there exists a dense *-subalgebra $\A^1$ of $\A$ such that for all $a\in \A^1$ and for all $X$ in the Lie algebra of $G$, the limit: 
\begin{equation}\label{partial-def}
\partial_X(a) =  \lim_{t\rightarrow 0} \frac{\alpha_{\exp(tX)}(a)-a}{t}
\end{equation}
exists in $\A$, where $\exp : \alg{g}\rightarrow G$ is the Lie local homeomorphism.

Moreover, by \cite{Hoegh-Krohn81}, if the fixed point C*-algebra of $\alpha$ is reduced to $\C\unit_\A$ --- i.e. if $\alpha$ is ergodic --- then there exists a unique tracial state $\tau$ on $\A$, and $\tau$ is moreover faithful. 
 
Rieffel proved in \cite{Rieffel98a} that the natural spectral triple one may construct from an ergodic action of a Lie group on a unital C*-algebra $\A$ provides a {\Lqcms} structure to $\A$. With these preliminaries established, the main theorem of our paper follows, and provides a certain form of perturbation of spectral triples constructed from ergodic actions of Lie groups along the lines proposed by \cite{Sitarz13,Sitarz15}.

\begin{theorem}\label{main-thm}
Let $\A$ be a unital C*-algebra, and let $\alpha$ be a strongly continuous ergodic action of a compact Lie group $G$ on $\A$. Let $n$ be the dimension of $G$. We endow the dual $\alg{g}'$ of the Lie algebra $\alg{g}$ of $G$ with an inner product $\inner{\cdot}{\cdot}$, and we denote by $C$ the Clifford algebra of $(\alg{g}',\inner{\cdot}{\cdot})$. Let $c$ be a faithful nondegenerate representation of $C$ on some Hilbert space $\Hilbert_C$.

We fix some orthonormal basis $\{e_1,\ldots,e_n\}$ of $\alg{g'}$, and we let $X_1,\ldots,X_n \in \alg{g}$ be the dual basis. For each $j\in \{1,\ldots,n\}$, we denote the derivation of $\A$ defined via $\alpha$ by $X_j$, using Equation (\ref{partial-def}) by $\partial_{X_j} = \partial_j$. Let $\A^1$ be the common domain of $\partial_1,\ldots,\partial_n$, which is a dense *-subalgebra in $\A$.

Let $\tau$ be the unique $\alpha$-invariant tracial state of $\A$. Let $\rho$ be the representation of $\A$ obtained from the Gelfan'd-Naimark-Segal construction applied to $\tau$ and let $L^2(\A,\tau)$ be the corresponding Hilbert space. As $\A^1$ is dense in $L^2(\A,\tau)$, the operator $\partial_j$ defines an unbounded densely defined operator on $L^2(\A,\tau)$ for all $j \in \{1,\ldots,n\}$.

Let $\Hilbert = L^2(\A,\tau) \overline{\otimes} \Hilbert_C$ where $\overline{\otimes}$ is the standard tensor product for Hilbert spaces. We define the following representation of $\A$ on $\Hilbert$:
\begin{equation*}
\pi(a) : b \otimes f \longmapsto \rho(a)b \otimes f\text{.}
\end{equation*}

Let:
\begin{equation*}
H = \begin{pmatrix}
h_{11} & \cdots & h_{1n}\\
\vdots & & \vdots \\
h_{n1} & \cdots & h_{nn}
\end{pmatrix}
\end{equation*}
where for all $j, k \in \{1,\ldots,n\}$, the coefficients $h_{jk}$ are elements in the commutant of $\A$ in $L^2(\A,\tau)$, and where $H$ is invertible as an operator on $\Hilbert' = \oplus_{j=1}^n L^2(\A,\tau)$. We denote the identity over $\Hilbert'$ by $\unit_{\Hilbert'}$.

We define:
\begin{equation*}
D_H = \sum_{j=1}^n \sum_{k=1}^n h_{kj} \partial_k \otimes c(e_j)
\end{equation*}
so that for all $a\in\A$:
\begin{equation*}
[D_H,\pi(a)] = \sum_{j=1}^n \sum_{k=1}^n h_{kj} \rho(\partial_k(a)) \otimes c(e_j)\text{.}
\end{equation*}
We define:
\begin{equation*}
\Lip_H(a) = \opnorm{[D_H,\pi(a)]}_{\Hilbert}\text{.}
\end{equation*}

Then:
\begin{enumerate}
\item $(\A,\Lip_H)$ is a {\Lqcms},
\item if we set:
\begin{equation*}
H' = \begin{pmatrix}
h_{11}' & \cdots & h_{1n}'\\
\vdots & & \vdots \\
h_{n1}' & \cdots & h_{nn}'
\end{pmatrix}
\end{equation*}
where $h_{jk}'$ lies in the commutant of $\rho(\A)$ for all $j,k \in \{1,\ldots,n\}$ and where $H'$ is invertible as an operator on $\Hilbert'$, then:
 \begin{multline*}
\qpropinquity{}((\A,\Lip_H),(\A,\Lip_{H'})) \leq  n \max\left\{ \opnorm{\unit_{\Hilbert'} - H' H^{-1}}_{\Hilbert'}, \opnorm{\unit_{\Hilbert'}-H H'^{-1}}_{\Hilbert'} \right\}\\ \times \left[1 + 
\frac{1}{2}\max\left\{\left( 1 + n\opnorm{1-H^{-1}}_{\Hilbert'}\right)^{-1},\right.\right. \\ \left.\left. \left( 1 + n\opnorm{1-H'^{-1}}_{\Hilbert'}\right)^{-1}\right\}\diam{\StateSpace(\A)}{\Kantorovich{\Lip}} \right]\text{.}
\end{multline*}
\end{enumerate}
\end{theorem}

\begin{proof}
We observe that if $H = \unit_{\Hilbert'}$ then $D_H$ (which we will denote simply by $D$) is the operator constructed in \cite{Rieffel98a}, and the associated seminorm $\Lip : a\in\sa{\A} \mapsto \opnorm{[D,\pi(a)]}_{\Hilbert}$ is a Lip-norm. By \cite{Rieffel99}, $\Lip$ is closed as well. In conclusion, $(\A,\Lip)$ is a {\Lqcms}.

We identify $\A$ with a dense subspace of $L^2(\A,\tau)$ as $\tau$ is faithful.

Let $a\in\A$ in the domain $\A^1$ of all the derivations $\partial_1,\ldots,\partial_n$. We check that, for any $b\otimes f \in \Hilbert$ with $b\in\A^1$, we have, using the fact $h_{jk}$ commutes with $\rho(a)$ for all $j,k \in \{1,\ldots,n\}$:
\begin{equation*}
\begin{split}
[D_H,\pi(a)](b\otimes f) &= \sum_{j=1}^n \left(\sum_{k=1}^n h_{kj} (\partial_k (ab) - \rho(a)\partial_k(b)) \right)\otimes c(e_j)f\\
&=\sum_{j=1}^n \left(\sum_{k=1}^n h_{kj} (\rho(a)\partial_k (b) + \rho(\partial_k(a))b - \rho(a)\partial_k(b)) \right)\otimes c(e_j)f\\
&= \sum_{j=1}^n\left(\sum_{k=1}^n h_{kj}\rho(\partial_k(a))\right)\otimes c(e_j)f \text{.}
\end{split}
\end{equation*}
Thus $[D_H,\pi(a)] = \sum_{j=1}^n\left(\sum_{k=1}^n h_{kj}\rho(\partial_k(a))\right)\otimes c(e_j)$ as claimed, since the vectors of the form $b\otimes f$ with $b\in \A^1$ and $f\in C$ span a dense subspace of $\Hilbert$, which we take as the domain of $D_H$. We remark that this domain is independent of $H$.  Moreover, $[D_H,\pi(a)]$ is a bounded operator on $\Hilbert$ as a finite sum of bounded operators, and since $\A^1$ is dense in $\A$, we conclude that $\Lip_H$ is a seminorm defined on a dense Jordan-Lie subalgebra of $\sa{\A}$, namely the space $\sa{\A^1}$ of all self-adjoint elements in $\A^1$.

We now work with a general $H$ and $H'$ as given in the assumptions of our theorem.

Let now $a\in\sa{\A^1}$. We wish to estimate $|\Lip_H(a) - \Lip_{H'}(a)|$. To this end, we will find an upper bound for the operator norm of $[D_H - D_{H'},\pi(a)]$ in terms of a different operator on a different Hilbert space, by a succession of computations which we now present.

We begin with:
\begin{equation}\label{eq0}
\begin{split}
|\Lip_H(a) - \Lip_{H'}(a)| &= \opnorm{\sum_{j=1}^n \left(\sum_{k=1}^n (h_{kj}-h_{kj}')\rho(\partial_k(a)) \right)\otimes c(e_j)}_{\Hilbert} \\
&\leq \sum_{j=0}^n \opnorm{\left(\sum_{k=1}^n (h_{kj}-h_{kj}')\rho(\partial_k(a))\right)\otimes c(e_j)}_{\Hilbert}\\
&\leq \sum_{j=0}^n\opnorm{\sum_{k=1}^n(h_{kj}-h'_{kj})\rho(\partial_k(a))}_{L^2(\A,\tau)}\opnorm{c(e_j)}_{\Hilbert_C}\\
&= \sum_{j=0}^n\opnorm{\sum_{k=1}^n(h_{kj}-h'_{kj})\rho(\partial_k(a))}_{L^2(\A,\tau)}\text{.}
\end{split}
\end{equation}
Note that $c(e_j)$ is a unitary on $\Hilbert_{C}$ for all $j\in \{1,\ldots,n\}$, so indeed $\opnorm{c(e_j)}_{\Hilbert_C} = 1$.

Let $\varepsilon > 0$ be given. For any $j\in \{1,\ldots,n\}$ there exists $\xi_j\in L^2(\A,\tau)$ with $\|\xi_j\|_{L^2(\A,\tau)} = 1$ and:
\begin{equation*}
\begin{split}
\left\|\sum_{k=1}^n(h_{kj}-h_{kj}')\rho(\partial_k(a)) \xi_j\right\|_{L^2(\A,\tau)} &\leq \opnorm{\sum_{k=1}^n(h_{kj}-h_{kj}')\rho(\partial_k(a))}_{L^2(\A,\tau)}\\
&\leq \left\|\sum_{k=1}^n(h_{kj}-h_{kj}')\rho(\partial_k(a)) \xi_j\right\|_{L^2(\A,\tau)} + \frac{\varepsilon}{n} \text{.}
\end{split}
\end{equation*}
Thus, continuing from Inequality (\ref{eq0}):
\begin{equation}\label{eq1}
\begin{split}
|\Lip_H(a) - \Lip_{H'}(a)| &\leq \sum_{j=0}^n\left\|\sum_{k=1}^n(h_{kj}-h'_{kj})\rho(\partial_k(a))\xi_j\right\|_{L^2(\A,\tau)} + \varepsilon\\
&\leq \sqrt[2]{n}\sqrt[2]{\sum_{j=1}^n \left\|\sum_{k=1}^n(h_{kj}-h_{kj}')\rho(\partial_k(a)) \xi_j\right\|^2_{L^2(\A,\tau)}} + \varepsilon\text{.}
\end{split}
\end{equation}

Now, let $\Hilbert' = \oplus_{j=1}^n L^2(\A,\tau)$ and let:
\begin{equation*}
\partial(a) = \begin{pmatrix}
\rho(\partial_1(a)) & & \\
& \ddots & & \\
& & \rho(\partial_n(a))
\end{pmatrix}\text{.}
\end{equation*}
For any operator:
\begin{equation*}
T = \begin{pmatrix}
t_{11} & \cdots & t_{1n}\\
\vdots & & \vdots\\
t_{n1} & \cdots & t_{nn}
\end{pmatrix}\text{,}
\end{equation*}
where for all $j,k \in \{1,\ldots,n\}$, the coefficient $t_{jk}$ is an operator on $L^2(\A,\tau)$, and for any $\eta = (\eta_1,\ldots,\eta_n) \in \Hilbert'$, we have:
\begin{equation*}
T\partial(a)\eta = T\begin{pmatrix}
\rho(\partial_1(a))\eta_1\\
\vdots\\
\rho(\partial_n(a))\eta_n
\end{pmatrix}
= \begin{pmatrix}
\sum_{k=1}^n t_{k1}\rho(\partial_k(a)) \\
\vdots\\
\sum_{k=1}^n t_{kn}\rho(\partial_k(a))
\end{pmatrix}
\end{equation*}
and thus:
\begin{equation*}
\|T \partial(a)\eta\|_{\Hilbert'} = \sqrt{\sum_{j=1}^n\left\|\sum_{k=1}^n t_{kj}\rho(\partial_k(a))\eta_j\right\|_{L^2(\A,\tau)}^2} \text{.}
\end{equation*}
Thus, in particular:
\begin{multline}\label{eq2b}
\sqrt[2]{\sum_{j=1}^n \left(\left\|\sum_{k=1}^n(h_{kj}-h_{kj}')\rho(\partial_k(a)) \xi_j\right\|^2_{L^2(\A,\tau)}\right)} \\
\begin{aligned}
&= \|(H-H')\partial(a) \xi\|_{\Hilbert'}\\
&= \|(1-H'H^{-1})H \partial(a)\xi\|_{\Hilbert'}\\
&\leq \opnorm{1-H' H^{-1}}_{\Hilbert'} \|H\partial(a) \xi\|_{\Hilbert'}\\
&\leq \opnorm{1-H' H^{-1}}_{\Hilbert'} \sqrt[2]{\sum_{j=1}^n\left\|\sum_{k=1}^n h_{kj}\rho(\partial_k(a))\right\|_{L^2(\A,\tau)}^2 \|\xi_j\|_{L^2(\A,\tau)}^2}\\
&\leq \opnorm{1-H' H^{-1}}_{\Hilbert'} \sqrt[2]{n} \max\left\{ \opnorm{\sum_{k=1}^n h_{kj}\rho(\partial_k(a))}_{L^2(\A,\tau)} : j \in \{1,\ldots,n\} \right\}\text{.}
\end{aligned}
\end{multline}

Let $j \in \{1,\ldots,n\}$ and let $p_j = \frac{1+c(e_j)}{2}$ and $q_j = \frac{1-c(e_j)}{2}$. As in \cite{Rieffel98a}, since $c(e_j)$ is an involutive unitary, we note that by construction, $p_j$ and $q_j$ are projections; furthermore $p_j e_k p_j = \delta_j^k p_j$ and $q_j e_k q_j = \delta_j^k q_j$ for all $k\in\{1,\ldots,n\}$, where $\delta_j^k$ is the Kroenecker symbol (note: in \cite{Rieffel98a}, Rieffel used the Clifford algebra for $-\inner{\cdot}{\cdot}$ and thus our $e_j$ is his $i e_j$).

Now, let $x  = \sum_{j=1}^n \rho(a_j) \otimes c(e_j)$ for some $a_1,\ldots,a_n \in \A$, so $x$ is an operator on $\Hilbert$. Now for any $j\in\{1,\ldots,n\}$:
\begin{equation*}
(\unit_\A \otimes p_j)x(\unit_\A \otimes p_j) = \rho(a_j)\otimes p_j
\end{equation*}
and
\begin{equation*}
(\unit_\A \otimes q_j)x(\unit_\A \otimes q_j) = \rho(a_j)\otimes q_j
\end{equation*}
and thus, since $\unit_\A\otimes p_j, \unit_\A\otimes q_j \leq \unit_\Hilbert$, we get:
\begin{equation*}
\max\left\{\opnorm{\rho(a_j)\otimes p_j}_\Hilbert, \opnorm{\rho(a_j)\otimes q_j}_\Hilbert\right\}\leq\opnorm{x}_\Hilbert\text{.}
\end{equation*}
Therefore, for any $j\in\{1,\ldots,n\}$, since either $p_j$ or not $q_j$ is nonzero as $p_j+q_j = c(\unit_{C})$, and $\rho$ is an isometry:
\begin{equation*}
\|a_j\|_\A \leq \opnorm{x}_{\Hilbert}\text{.}
\end{equation*}
Applying this observation to:
\begin{equation*}
[D_H,\pi(a)] = \sum_{j=1}^n \left(\sum_{k=1}^n h_{kj}\rho(\partial_k(a))\right)\otimes c(e_j)
\end{equation*}
where $a\in\sa{\A}$, we obtain:
\begin{equation}\label{eq2c}
\Lip_H(a) = \opnorm{[D,\pi(a)]}_{\Hilbert} \geq \max\left\{ \opnorm{\sum_{k=1}^n h_{kj}\rho(\partial_k(a))}_{\Hilbert} \right\}\text{.}
\end{equation}

We thus apply Inequality (\ref{eq2c}) to the last expression of Inequality (\ref{eq2b}) to obtain:
\begin{equation}\label{eq2}
\sqrt[2]{\sum_{j=1}^n \left(\left\|\sum_{k=1}^n(h_{kj}-h_{kj}')\rho(\partial_k(a)) \xi_j\right\|^2_{L^2(\A,\tau)}\right)} \leq \opnorm{1-H' H^{-1}}_{\Hilbert'} \sqrt[2]{n}\Lip_H(a)\text{.}
\end{equation}

Thus, stringing Inequalities (\ref{eq1}) and (\ref{eq2}) together, we obtain that, for all $\varepsilon > 0$:
\begin{equation*}
|\Lip_{H'}(a) - \Lip_{H}(a)| \leq  \varepsilon + n \opnorm{1-H' H^{-1}}_{\Hilbert'} \Lip_H(a) \text{,}
\end{equation*}
and thus:
\begin{equation}\label{estimate-eq}
|\Lip_{H'}(a) - \Lip_{H}(a)| \leq  n \opnorm{1-H' H^{-1}}_{\Hilbert'} \Lip_H(a) \text{.}
\end{equation}

A symmetric computation would show that:
\begin{equation}\label{estimate3-eq}
|\Lip_{H'}(a) - \Lip_{H}(a)| \leq  n \opnorm{1-H H'^{-1}}_{\Hilbert'} \Lip_{H'}(a) \text{.}
\end{equation}

Now, we make several observations. To begin with, we note that, from Inequality (\ref{estimate-eq}),  for all $a\in \A$ we have:
\begin{equation}\label{estimate2-eq}
\Lip_{H'}(a) \leq \left(n\opnorm{1-H' H^{-1}}_{\Hilbert'}+1\right)\Lip_H(a)\text{.}
\end{equation}

Now, if $H'$ is the identity $\unit_{\Hilbert'}$ of $\Hilbert'$, then $\Lip_{H'} = \Lip$ is a Lip-norm, as discussed at the beginning of this proof. By Lemma (\ref{comp-lemma}) and Estimate (\ref{estimate2-eq}), we conclude that $\Lip_{H}$ is also a Lip-norm, and moreover:
\begin{equation}\label{estimate4-eq}
\diam{\StateSpace(\A)}{\Kantorovich{\Lip_H}} \leq \left(n\opnorm{1-H^{-1}}_{\Hilbert'}+1\right)^{-1}\diam{\StateSpace(\A)}{\Kantorovich{\Lip}}\text{.}
\end{equation}

Next, we note that Estimates (\ref{estimate-eq}) and (\ref{estimate3-eq}) meet the assumptions of Lemma (\ref{propinquity-estimate-lemma}): for all $a\in\sa{\A}$, we have:
\begin{equation*}
\max\left\{ \|a\unit_\A - \unit_\A a\|_\A, |\Lip_H(a)-\Lip_{H'}(a)| \right\} \leq n \opnorm{1-H' H^{-1}}_{\Hilbert'} \Lip_H(a) \text{,} 
\end{equation*}
and symmetrically in $H$,$H'$.

Putting together Estimates (\ref{estimate-eq}), (\ref{estimate3-eq}),  (\ref{estimate2-eq}) and (\ref{estimate4-eq}) with Lemma (\ref{propinquity-estimate-lemma}), we establish that:
\begin{multline*}
\qpropinquity{}((\A,\Lip_H),(\A,\Lip_{H'})) \leq 
n \max\left\{ \opnorm{\unit_{\Hilbert'} - H' H^{-1}}_{\Hilbert'}, \opnorm{\unit_{\Hilbert'}-H H'^{-1}}_{\Hilbert'} \right\}\\ \times \left[1 + \frac{1}{2}\max\left\{\left( 1 + n\opnorm{1-H^{-1}}_{\Hilbert'}\right)^{-1},\right.\right. \\ \left.\left.\left( 1 + n\opnorm{1-H'^{-1}}_{\Hilbert'}\right)^{-1}\right\}\diam{\StateSpace(\A)}{\Kantorovich{\Lip}} \right]\text{,}
\end{multline*}
where we used $(\alg{B}(\Hilbert), \pi, \pi, \unit_{\alg{B}(\Hilbert)})$ as our bridge, whose height is null.
\end{proof}

\begin{corollary}
We shall use the same notations as in Theorem (\ref{main-thm}). Let $\alg{D} = M_n(\rho(\A)')$ be the C*-algebra of $n\times n$ matrices over the commutant $\rho(\A)'$ of $\rho(\A)$. Endow the group $\mathrm{GL}(\D)$ of invertible elements in $\D$ with the length function:
\begin{equation*}
\ell(K) = \| K - \unit_\D\|_{\D} = \opnorm{K-\unit_{\Hilbert'}}_{\Hilbert'} \text{,}
\end{equation*}
and endow $\mathrm{GL}(\D)$ with the topology induced by $\ell$. If $H\in\mathrm{GL}(\D)$ then:
\begin{equation*}
\lim_{\substack{K\longrightarrow H \\
 K \in (\mathrm{GL}(\D),\ell)}} \qpropinquity{}((\A,\Lip_K),(\A,\Lip_H)) = 0 \text{.}
\end{equation*}
\end{corollary}

\begin{proof}
Let $H\in\mathrm{GL}(\D)$ and let $(H_k)_{k\in\N}$ be a sequence in $\mathrm{GL}(\D)$ such that:
\begin{equation*}
\lim_{k\rightarrow\infty} \ell(H_k^{-1} H) = 0\text{,}
\end{equation*}
i.e. $\lim_{k\rightarrow\infty} \|1-H_k^{-1}H\|_\D = 0$ and thus by continuity, $\lim_{k\rightarrow\infty} \|1-H^{-1} H_k\|_\D = 0$. Note that by definition, $\opnorm{\cdot}_{\Hilbert'} = \|\cdot\|_\D$. Now, we have $\lim_{k\rightarrow \infty} \|H_k^{-1}-H^{-1}\|_\D = 0$ and thus $(H_k^{-1})_{k\in\N}$ is bounded.

Thus we may apply Theorem (\ref{main-thm}). We note that $\left(\opnorm{H_k^{-1}-\unit_\D}\right)_{k\in\N}$ is bounded, so:
\begin{multline*}
\left(1 + \frac{1}{2}\max\left\{ \left(1 + n\opnorm{1-H^{-1}}_{\Hilbert'}\right)^{-1}, \left(1 + n \opnorm{1-H_k^{-1}}_{\Hilbert'}\right)^{-1}\right\}\diam{\StateSpace(\A)}{\Kantorovich{\Lip}} \right)_{k\in\N}
\end{multline*}
is bounded as well (note that $\Kantorovich{\Lip}$ does not depend on $k$, or even $H$).

Thus we conclude:
\begin{equation*}
\lim_{n\rightarrow\infty} \qpropinquity{}((\A,\Lip_{H_n}),(\A,\Lip_H)) = 0\text{,}
\end{equation*}
and our corollary is proven.
\end{proof}

We conclude our paper by stating that Theorem (\ref{main-thm}) and its corollary above may be applied to the setting of \cite{Sitarz13,Sitarz15}. We fix $d \in \{2,3,\ldots\}\subseteq\N$. Let $\Theta = \left(\theta_{jk}\right)_{j,k \in \{1,\ldots,d\}}$ be a $d\times d$ antisymmetric matrix with entries in $[0,1]$. For any two $\xi,\eta \in \Z^d$, we define:
\begin{equation*}
\sigma(\xi,\eta) = \exp(i\pi \inner{\xi}{\Theta\eta})
\end{equation*}
where $\inner{\cdot}{\cdot}$ is the usual standard inner product of $\C^d$.

Let $\A_\Theta = C^\ast(\Z^d,\sigma)$ be the twisted group C*-algebra of $\Z^d$, i.e. the quantum torus for the matrix $\Theta$. In other words, $\A_\Theta$ is the universal C*-algebra generated by $d$ unitaries $U_1,\ldots, U_d$ such that:
\begin{equation*}
U_j V_k = \exp(2i\pi\theta_{jk}) V_k U_j \text{.}
\end{equation*}

The $d$-torus:
\begin{equation*}
\T^d = \{(z_j)_{j\in\{1,\ldots,d\}} \in\C^d  : \forall j \in \{1,\ldots,d\} \quad |z_j|=1\}
\end{equation*}
acts ergodically on $\A_\Theta$ by setting, for all $\lambda = (\lambda_j)_{j\in\{1,\ldots,d\}}\in\T^d$ and $j \in \{1,\ldots,d\}$:
\begin{equation*}
\alpha_{\lambda}(U_j) = \lambda_j U_j \text{.}
\end{equation*}

As the dual action $\alpha$ of $\T^d$ on $\A_\Theta$ is ergodic, there exists a unique $\alpha$-invariant tracial state $\tau$ on $\A_\Theta$. Let $\rho$ be the GNS representation of $\A_\Theta$ on $L^2(\A_\Theta,\tau)$ defined by $\tau$. To clarify our notations later on, it will be useful to employ the notation:
\begin{equation*}
\xi : \begin{cases}
\A_\Theta &\longrightarrow L^2(\A_\Theta,\tau)\\
a &\longmapsto a \text{.}
\end{cases}
\end{equation*}

Since $\tau$ is $\alpha$-invariant, the *-automorphism $\alpha_\lambda$, for any $\lambda\in\T^d$, extends to an isometry $V_\lambda$ on $L^2(\A_\Theta,\tau)$. Moreover, we define the anti-linear isometry $J$ on $L^2(\A_\Theta,\tau)$ by extending by continuity the function defined for all $a\in\A_\Theta$ by $J\xi(a) = \xi(a^\ast)$.

Let $a\in \A_\Theta$. We check that for all $\lambda\in\T^d$:
\begin{equation*}
\begin{split}
V_\lambda J \xi(a) &= V_\lambda \xi(a^\ast) = \xi(\alpha_\lambda(a^\ast)) \\
&= \xi(\alpha_\lambda(a)^\ast) = J\xi(\alpha_\lambda(a)) = JV_\lambda \xi(a)  \text{.}
\end{split}
\end{equation*}
Thus $JV_\lambda = V_\lambda J$ as $\A_\Theta$ is dense in $L^2(\A_\Theta,\tau)$. A simple computation also shows that for all $a, b\in\A_\Theta$ and $\lambda\in\T^d$:
\begin{equation*}
\begin{split}
V_\lambda\rho(a)V_\lambda^{\ast} \xi(b) &= V_\lambda \rho(a)\xi(\alpha_{\lambda^{-1}}(b)) \\
&= \xi(\alpha_\lambda(a\alpha_{\lambda^{-1}}(b))) = \xi(\alpha_\lambda(a)b) \\
&= \rho(\alpha_\lambda(a))\xi(b) \text{.}
\end{split}
\end{equation*}
Thus $\mathrm{Ad}_{V_\lambda}(\rho(a)) = \rho(\alpha_\lambda(a))$ for all $a\in\A_\Theta$, $\lambda\in\T^d$. Similarly, of course, $J\rho(a)J = \rho(a^\ast) = \rho(a)^\ast$ for all $a\in\A_\Theta$.

Consequently, for all $a\in\A_\Theta$ and $\lambda\in\T^d$:
\begin{equation}\label{strongly-cont-eq}
\begin{split}
\alpha_\lambda(J\rho(a)J) - J\rho(a)J &= V_\lambda J\rho(a)J V_\lambda^\ast -J\rho(a)J \\
&= J(V_\lambda \rho(a) V_\lambda^\ast - \rho(a) )J \\
&= J\rho(\alpha_\lambda(a) - a)J \text{.}
\end{split}
\end{equation}
Since $\alpha$ is strongly continuous on $\A_\Theta$, we conclude from Expression (\ref{strongly-cont-eq}) that $\lambda \in \T^d \mapsto \mathrm{Ad}_\lambda$ is also strongly continuous on $J\rho(\A_\Theta)J$.

On the other hand, we note that $J\rho(\A_\Theta)J$ lies in the commutant of $\rho(\A_\Theta)$ on $L^2(\A_\Theta,\tau)$. Let us denote $J\rho(\A_\Theta)J$ simply by $\B_\Theta$.

For any $j \in \{1,\ldots,d\}$, we let $X_j$ be the operator $\frac{\partial}{\partial z_j}$ on functions on $\T^d$, i.e. the operator such that, for $f \in C(\T^d)$ for which the limit exists, we have:
\begin{equation*}
X_j f (z_1,\ldots,z_d) = \lim_{t\rightarrow 0} \frac{f\left(z_1.\ldots,z_{j-1},e^{it}, z_{j+1},\ldots,z_d\right) - f(1,\ldots,1) }{t}\text{,}
\end{equation*}
for all $(z_1,\ldots,z_d) \in \T^d$.

Thus, $\left\{X_1,\ldots,X_d\right\}$ is a basis for the Lie algebra $\mathfrak{t}_d$ of $\T^d$. Since $\alpha$ is a strongly continuous action of $\T^d$ on $\A_\Theta$, by \cite[Proposition 3.6.1]{Bratteli79}, there exists a dense *-subalgebra $\A_\Theta^1$ of $\A_\Theta$ and, for all $j\in\{1,\ldots,d\}$, a derivation $\delta_j : \A_\Theta^1 \rightarrow \A_\Theta^1$ such that:
\begin{equation}\label{derivation-eq}
\partial_j a = \lim_{t\rightarrow e} \frac{\alpha_{\exp(tX_j)}(a)-a}{t} = \alpha_{X_j}(a)\text{,}
\end{equation}
for all $a\in\A_\Theta^1$. We define, with a slight abuse of notation, the operators $\delta_j(\rho(a)) = \rho(\delta_j(a))$ for all $a\in \A_\Theta^1$.

Moreover, since $\T^d$ acts on $\B_\Theta = J\rho(\A_\Theta)J$, we may also find a dense *-subalgebra  $\B_\Theta^1$ of $\B_\Theta$ and derivations $\partial_1',\ldots,\partial_d'$ defined by a similar expression as Expression (\ref{derivation-eq}). As there will be no risk for confusion, we will also denote $\partial_j'$ by $\partial_j$.

To make this picture complete, we note that since $J$ commutes with $V_\lambda$ for all $\lambda\in\T^d$, we have for all $\B_\Theta^1 = J\rho(\A_\Theta)J$ and $\partial_j(b) = J\partial_j(\rho(a))J$ for all $b \in \B_\Theta^1$ and $b=J\rho(a)J$. Consequently, we also note that $\partial_j(b)$ commutes with $\rho(\A_\Theta)$ for all $b\in \B_\Theta^1$.

We now set $\inner{\cdot}{\cdot}$ so that $\{X_1,\ldots,X_d\}$ is orthonormal. For each $j\in \{1,\ldots,d\}$, we set $e_j = \inner{X_j}{\cdot}$ so that $\{e_1,\ldots,e_d\}$ is the dual basis of $\{X_1,\ldots,X_d\}$ in $\alg{t}_d'$, and we endow $\alg{t}_d'$ with the inner product which makes $\{e_1,\ldots,e_d\}$ orthonormal. We let $c$ be a unital faithful *-representation of the Clifford algebra $\alg{C}_d$ of $\alg{t}_d'$ on $\C^d$. In particular, if $d=2$, we recover \cite{Sitarz13,Sitarz15} by setting: $c(1)$ the identity, $c(e_1)=\sigma_1$, $c(e_2)=\sigma_2$ and $c(e_1e_2) = \sigma_3$ where $\sigma_1$,$\sigma_2$ and $\sigma_3$ are the Pauli matrices:
\begin{equation*} 
\sigma_1 = \begin{pmatrix}
0 & 1 \\
1 & 0 
\end{pmatrix}
\text{, }
\sigma_2 = \begin{pmatrix}
0 & -i \\
i & 0 
\end{pmatrix}
\text{ and }
\sigma_3 = \begin{pmatrix}
1 & 0 \\
0 & -1 
\end{pmatrix}\text{.}
\end{equation*}

We now define $\pi$ as the representation of $\A_\Theta$ on $\Hilbert = L^2(\A_\Theta,\tau)\otimes \C^d$ defined by $\pi(a)(\xi\otimes z) = \rho(a)\xi\otimes z$ for all $a\in\A_\Theta$, $\xi \in L^2(\A,\tau)$, $z\in \C^d$, and extended by linearity.

In \cite[Equation (2.9)]{Sitarz13}, D\k{a}browski and Sitarz worked with the family of spectral triples whose Dirac operator is of the form:
\begin{equation*}
D_H' = \sum_{j=1}^d\left(\sum_{k=1}^d h_{kj}\partial_k + \frac{1}{2}\partial_k(h_{kj})\right)\otimes\sigma_j
\end{equation*}
where $H=\begin{pmatrix} h_{11} &\cdots & h_{1d} \\ \vdots & & \vdots \\ h_{d1} & \cdots & h_{dd} \end{pmatrix}$ and $h_{jk}$ lie in $\B_\Theta = J\rho(\A_\Theta)J$ for all $j,k \in \{1,\ldots,d\}$. 

Now $\partial_j(h_{jk})$ commute with $\rho(\A_\theta)$ for all $j,k \in \{1,\ldots,d\}$ and thus, using the notations of Theorem (\ref{main-thm}), we have:
\begin{equation*}
[D_H,\pi(a)] = [D_H',\pi(a)]
\end{equation*}
for all $a\in\A_\Theta'$.

Thus, our work applies equally well to spectral triples of the form $(\A_\theta,\Hilbert,D_H')$ as of the form $(\A_\theta,\Hilbert,D_H)$. We now have established that we can apply our Theorem (\ref{main-thm}) to the curved quantum tori described by the triples $(\A_\Theta,L^2(\A_\Theta)\otimes\C^d,D_H)$, and this concludes our study of the metric properties of the curved quantum tori in this paper.

\providecommand{\bysame}{\leavevmode\hbox to3em{\hrulefill}\thinspace}
\providecommand{\MR}{\relax\ifhmode\unskip\space\fi MR }
\providecommand{\MRhref}[2]{%
  \href{http://www.ams.org/mathscinet-getitem?mr=#1}{#2}
}
\providecommand{\href}[2]{#2}

\vfill

\end{document}